\documentclass[letterpaper, 10pt, conference]{ieeeconf}

\IEEEoverridecommandlockouts                              

\usepackage[bmargin=25.4mm, rmargin=19.1mm, lmargin=19.1mm, tmargin=19.1mm]{geometry}
\usepackage{graphicx}
\usepackage{amsmath}
\usepackage{amssymb}

\usepackage{dsfont}
\usepackage{caption}
\usepackage{subcaption}
\usepackage{url}

\usepackage{amsthm}
\newtheorem{theorem}{Theorem}

\newtheorem{lemma}[theorem]{Lemma}

\newtheorem{remark}{Remark}

\usepackage{xcolor}

\newcommand{\bal}[1] {\ensuremath{\left(\begin{array}{#1}}}
\newcommand{\ear} {\ensuremath{\end{array}\right)}}

\newcommand{\bals}[1] {\ensuremath{\left[\begin{array}{#1}}} 
\newcommand{\ears} {\ensuremath{\end{array} \right] }} 

\let\leq\leqslant
\let\geq\geqslant

\newcommand{\calC}{\ensuremath{\mathcal{C}}}

\newcommand{\calE}{\ensuremath{\mathcal{E}}}

\newcommand{\calG}{\ensuremath{\mathcal{G}}}
\newcommand{\calH}{\ensuremath{\mathcal{H}}}

\newcommand{\calL}{\ensuremath{\mathcal{L}}}

\newcommand{\calV}{\ensuremath{\mathcal{V}}}

\newcommand{\bmat}{\begin{matrix}}
\newcommand{\emat}{\end{matrix}}
\newcommand{\bbm}{\begin{bmatrix}}
\newcommand{\ebm}{\end{bmatrix}}
\newcommand{\bpm}{\begin{pmatrix}}
\newcommand{\epm}{\end{pmatrix}}
\newcommand{\bse}{\begin{subequations}}
\newcommand{\ese}{\end{subequations}}
\newcommand{\beq}{\begin{equation}}
\newcommand{\eeq}{\end{equation}}
\newcommand{\ben}{\begin{enumerate}}
\newcommand{\een}{\end{enumerate}}
\newcommand{\beni}{\renewcommand{\labelenumi}{\roman{enumi}.}
\renewcommand{\theenumi}{\roman{enumi}}\begin{enumerate}}
\newcommand{\eeni}{\end{enumerate}\renewcommand{\labelenumi}{\arabic{enumi}.}
\renewcommand{\theenumi}{\arabic{enumi}}}
\newcommand{\bena}{\renewcommand{\labelenumi}{\alpha{enumi}.}
\renewcommand{\theenumi}{\alpha{enumi}}\begin{enumerate}}
\newcommand{\eena}{\end{enumerate}\renewcommand{\labelenumi}{\arabic{enumi}.}
\renewcommand{\theenumi}{\arabic{enumi}}}
\newcommand{\bit}{\begin{itemize}}
\newcommand{\eit}{\end{itemize}}

\newcommand{\R}{\ensuremath{\mathbb R}}

\title{\LARGE \bf
Multi-agent deployment under the leader displacement \\
measurement: a PDE-based approach}

\author{Jieqiang Wei, Emilia Fridman, Anton Selivanov, Karl H. Johansson 
\thanks{*This work is supported by Knut and Alice Wallenberg Foundation, Swedish Research Council, and Swedish Foundation for Strategic Research.}
\thanks{  Jieqiang Wei, Anton Selivanov and Karl H. Johansson are with the Department of Automatic Control, School of Electrical Engineering and Computer Science. 
 KTH Royal Institute of Technology,
 SE-100 44 Stockholm, Sweden.
 {\tt\small \{jieqiang, kallej\}@kth.se, antonselivanov@gmail.com}.
 \newline
Emilia Fridman is with the School of Electrical Engineering, Tel Aviv
 University, Israel.
 {\tt\small emilia@eng.tau.ac.il}.
}
}

\begin{document}

\maketitle

\begin{abstract}\label{s:Abstract}
We study the deployment of a first-order multi-agent system  over a desired smooth curve in 3D space. We assume that the agents have access to the local information of the desired curve and their displacements with respect to their closest neighbors, whereas in addition a leader is able to measure his absolute displacement with respect to the desired curve. In this paper we consider the case that the desired curve is a closed $\calC^2$ curve and we assume that the leader transmit his measurement to other agents through a communication network. We start the algorithm with displacement-based formation control protocol. Connections from this ODE model to a PDE model (heat equation), which can be seen as a reduced model, are then established. The resulting closed-loop system is modeled as a heat equation with delay (due to the communication). The boundary condition is periodic since the desired curve is closed. By choosing appropriate controller gains (the diffusion coefficient and the gain multiplying the leader state), we can achieve any desired decay rate provided the delay is small enough. The advantage of our approach is in the simplicity of the control law and the conditions.  Numerical example illustrates the efficiency of the method.
\end{abstract}


\IEEEpeerreviewmaketitle

\section{Introduction}\label{s:Introduction}

Most of the existing work on multi-agent systems (MAS) consider interconnected agents modeled using ordinary differential equations (ODEs) or difference equations, and design the control for each agent depending either on global or local information. Besides these studies, there has been some work using partial differential equations (PDEs) to describe the spatial dynamics of multi-agent systems, e.g., \cite{FREUDENTHALER2016,fridman2014introduction,Frihauf2011,Qi2015,SERVAIS2014}. This approach is especially powerful when the number of the agents is large. One of the advantages of using PDE models for MAS is to reduce a high-dimensional ODE system to a single PDE. Reversely, given a desired PDE model, the corresponding performance and the control protocol for the individual agents (in ODE form) can be designed by proper discretization. In principle, this procedure is independent with respect to the number of agents, provided this number is large enough. 

In this paper, we consider a formation control problem which is referred to as deployment. This can be seen as a combination of a displacement-based and position-based formation control method. Each agent measures the relative positions (displacements) of its neighboring agents with respect to a global coordinate system. The desired formation is specified by the desired displacements between pair of agents. Then the agents, without knowing their absolute positions, achieve the desired formation by controlling the displacements of their neighboring agents. This will lead the agents to the desired formation up to a constant distance. As pointed out in \cite{oh2015survey}, in order to move the agents to the prescribed absolute positions, a small number of agents able to measure their absolute positions are needed. For existing ODE methods we refer to \cite{DEMARINA2017,Tanner2007} and the references within. 

Here we review some related work on the multi-agent deployment using PDE models. In \cite{Frihauf2011} and \cite{Qi2015}, the agents’ dynamics are modeled by reaction-advection-diffusion PDEs. By using the backstepping approach to boundary control, the agents are deployed onto families of planar curves and 2D manifolds, respectively. In \cite{MEURER2011}, the authors consider finite-time deployment of MAS into a planar formation, via predefined spatial-temporal paths, using a leader-follower architecture, i.e., boundary control. 
The same problem of deployment into planar curves using boundary control is considered in \cite{SERVAIS2014} and \cite{FREUDENTHALER2016} by using non-analytic solutions and a modified viscous Burger's equation, respectively. In \cite{Pilloni2016}, the authors proposed a boundary control law for a MAS, which is modeled as the heat equation, to achieve state consensus.

The main contributions of the paper is that we propose a framework which connects a ODE formation control protocol and a PDE model for the deployment of mobile agents onto arbitrary closed $\calC^2$ curves. Furthermore, in this framework we assume only leader measures its absolute position and use simple static output-feedback control. More precisely, the leader calculates its displacement with respect to the desired curve. 
Then the leader sends the value of its displacement to all the agents by using a communication network which results in time-varying delay due to sampling and communication \cite{fridman2014introduction}. The other agents, which are referred as followers, have access only to the local information of the desired curve and displacements with respect to their neighbors. Since the desired formation is a closed curve, the MAS is modeled as a diffusion equation with periodic boundary condition. The method used in this paper is based on \cite{FRIDMAN2012} and \cite{SELIVANOV2018} which deal with Dirichlet and mixed boundary conditions. We derive linear matrix inequality (LMI) conditions with arbitrary delay for desired convergence rate. Compared to the ODE MAS with communication delay, e.g., \cite{Li2013}, the LMI conditions derived in this paper are simpler with lower dimension, and they are always feasible.


The paper is organized as follows. In Section \ref{s:Preliminaries}, some useful inequalities are recalled. The MAS deployment problem using sampled control is formulated in Section \ref{s:prob-formulation}. The main results are included in Section \ref{s:periodic-nodelay} and \ref{s:periodic-delay}. In Section \ref{s:periodic-nodelay}, we derive LMI conditions to guarantee the deployment on the closed $\calC^2$ curve for the desired decay rate without communication delay. In Section \ref{s:periodic-delay}, the similar type of the result is obtained for the case with delay. Simulations are presented in Section \ref{s:simulations}. The paper is concluded in Section \ref{s:conclusion}.

\textbf{Notations.} With $\R_{\geq 0}$ we denote the set of  non-negative real numbers, respectively.  
$\calL_2(a,b)$ is the Hilbert space of square integrable functions $\phi(\xi),\xi\in[a,b]$ with the corresponding norm given as $\|\phi\|_{\calL_2} = \sqrt{\int_a^b z^2 d\xi}$. $\calH^1(a,b)$ is the Sobolev space of absolutely continuous scalar functions $\phi:[a,b]\rightarrow\R$ with $\frac{d\phi}{d \xi}\in \calL_2(a,b)$ 
. 
$\calH^2(a,b)$ is the Sobolev space of scalar functions $\phi:[a,b]\rightarrow \R$ with absolutely continuous $\frac{d\phi}{d\xi}$ and with $\frac{d^2\phi}{d\xi^2}\in \calL_2(a,b)$.


\section{Preliminaries}\label{s:Preliminaries}

\begin{lemma}[Wirtinger's inequality \cite{hardy1952inequalities}]\label{e:Wirtinger}
	For $f\in \calH^1(a,b)$, 
	\begin{equation*}
	\|f\|\le\frac{2(b-a)}{\pi}\|f'\|\quad\text{if $f(a)=0$ or $f(b)=0$.}		
	\end{equation*}
\end{lemma}
\begin{lemma}[Halanay's inequality, \cite{halanay1966differential,fridman2014introduction}]
Let $0< \delta_1 < 2 \delta_0$ and let $V:[t_0-\tau_M,\infty)\rightarrow [0,\infty)$ be an absolutely continuous function that satisfies 
\begin{align}\label{e:halanay-1}
\dot{V}(t) \leq -2\delta_0 V(t) + \delta_1 \sup_{-\tau_M\leq \theta \leq 0} V(t+\theta), \quad t\geq t_0.
\end{align}
Then 
\begin{align}
V(t) \leq e^{-2\delta (t-t_0)} \sup_{-\tau_M\leq \theta \leq 0} V(t_0+\theta), \quad t\geq t_0,
\end{align}
where $\delta>0$ is the unique positive solution of 
\begin{align}\label{e:halanay-alpha}
\delta = \delta_0 - \frac{\delta_1 e^{2\delta \tau_M}}{2}.
\end{align}
\end{lemma}

\section{Problem formulation}\label{s:prob-formulation}
We consider $N$ agents in $\R^3$ governed by 
\begin{align}
\dot{z}_i=u_i,\quad i\in\{1,\ldots,N\}, 
\end{align}
where $z_i\in\R^3$ are the states and $u_i\in\R^3$ are the control inputs. The goal is to deploy the agents on a given closed $\calC^2$ curve $\gamma\colon[0,2\pi]\rightarrow \R^3$.

Let us denote the ring graph with $N$ vertices as $\calG_\ell=\{\calV,\calE\}$, where $\calV=\{v_1,\ldots,v_N\}$ is the vertex set and $\calE=\{(v_i,v_{i+1}), i = 1,\ldots,N-1\}\cup\{(v_N,v_1)\}$ is the edge set.  As a typical formation control procedure, one assigns $N$ points on the curve, denoted as $\gamma(h),\ldots,\gamma(Nh)$, where $h=2\pi/N$. Consider the following displacement-based formation control protocol 
\begin{equation}\label{e:displacement}
\begin{aligned}
\dot{z}_i (t) = & a\frac{(z_{i-1}(t) - z_i(t))+(z_{i+1}(t) - z_i(t))}{h^2} \\
& - a\frac{(\gamma((i-1)h) -\gamma(ih))+(\gamma((i+1)h) - \gamma(ih))}{h^2}\\
& i = 1,\ldots, N.
\end{aligned}
\end{equation}
where $z_0 = z_N, z_{N+1}=z_1$, $z_i\in\R^3$ is the position of the agent $v_i$, and $a>0$, guarantees that all agents converge to the formation 
\begin{align}
E:=\{(z_1,\ldots,z_N) \mid z_i-z_j = \gamma(ih)-\gamma(jh) \},
\end{align} 
which is the desired curve up to constant translations \cite{oh2015survey}.

\begin{remark}
The implementation of the system \eqref{e:displacement} includes firstly the agents align the local coordination system, then the agent compare the displacement (to its neighbors) with respect to the desired displacement continuously. It can be proved that the formation of the agents converges to the formation given by desired displacements asymptotically up to a constant translation \cite{oh2015survey}.
\end{remark}

As suggested in \cite{Trecate2006},  when $N$ is large, the model \eqref{e:displacement} is an approximation of  
\begin{align}\label{e:displacement_PDE}
z_t(x, t) = a (z_{xx}(x,t) - \gamma_{xx}(x)).
\end{align}
By denoting the error $e(x,t)=z(x,t)-\gamma(x)$, the error dynamic of \eqref{e:displacement_PDE} is given as the following heat equation 
\begin{align}\label{e:error-dyn-boundarycontrol}
e_t(x,t) = a e_{xx}(x,t), \quad x\in[0,2\pi]. 
\end{align}  
Notice that the components of $e(x,t)\in\R^3$ are decoupled.

It can be seen that system \eqref{e:displacement_PDE} cannot drive the agents onto the desired curve $\gamma$, but up to a constant translation. In fact,  $z^*=\gamma+c$ is an equilibrium of system \eqref{e:displacement_PDE} for any constant $c$. This is consistent with the displacement-based formation control in \cite{oh2015survey}. In order to solve this problem, we shall employ additional control input to guarantee the convergence to the desired curve. More precisely, we assign leader agents who can measure the absolute positions of themselves and of their targets. 

%

Since the desired curve $\gamma$ is closed and is $\calC^2$, it is natural to consider the multi-agent system with periodic boundary condition 
\begin{equation}\label{e:periodic}
\begin{aligned}
z(0,t) & = z(2\pi,t) \\
z_{x}(0,t) & = z_{x}(2\pi,t).
\end{aligned}
\end{equation}
Furthermore, we assume, without loss of generality, that the leader is located at $x=\pi$
and it can measure $z(\pi,t)-\gamma(\pi)$ and send this information to the other agents through a communication network which results in a bounded time-varying delay.
The closed-loop system is given as 
\begin{align}\label{e:displacement-pde-delay-pre}
z_t = a (z_{xx} - \gamma_{xx}) - K(z(\pi,t_k-\eta_k) & -\gamma(\pi)), 
\end{align}
where $t\in[t_k,t_{k+1})$, $a>0$, $K>0$, $t_k$ is the updating time of the controller, and $\eta_k$ is the network-induced delay. Parameters $a$ and $K$ are the control gains. By using the time-delay approach to networked control systems \cite[Chapter~7]{fridman2014introduction}, we denote $\tau(t)=t-t_k+\eta_k$. Then the system \eqref{e:displacement-pde-delay-pre} can be presented as
\begin{align}\label{e:displacement-pde-delay}
z_t = a (z_{xx} - \gamma_{xx}) - K(z(\pi,t-\tau(t))-\gamma(\pi)).
\end{align}
Here $\tau(t)\leq \tau_M$, where $\tau_M$ is the sum of the maximum transmission interval and maximum network-induced delay.
We shall refer to \eqref{e:displacement-pde-delay} with boundary condition \eqref{e:periodic} as the system with periodic boundary condition. In this paper, we set $t_0=0$.
In this case, the error dynamics are given as 
\begin{align}
e_t & = a e_{xx}-K e(\pi,t-\tau(t)) \label{e:dynamic+delay} \\
& = a e_{xx} - K [e(x,t-\tau(t)) - \int_{\pi}^{x} e_\zeta(\zeta,t-\tau(t))d\zeta]. \nonumber
\end{align}
with boundary condition
\begin{equation}
\begin{aligned}
e(0,t) & = e(2\pi,t) \\
e_{x}(0,t) & = e_{x}(2\pi,t).
\end{aligned}\label{e:periodic_e}
\end{equation}
Consider the initial condition for \eqref{e:dynamic+delay}, \eqref{e:periodic_e} as
\begin{align}
e(x,t)=e(x,0), t<0. \label{e:initial-cond}
\end{align}
The stability of this system will be analyzed in Section \ref{s:periodic-delay}.

By defining 
\begin{align*}
X = & \{ w\in\calH^1(0,2\pi) \mid w(0,t) = w(2\pi,t) \},
\end{align*}
the existence and uniqueness of the strong solution of system \eqref{e:dynamic+delay} with periodic boundary condition \eqref{e:periodic_e} is guaranteed by the arguments in \cite{SELIVANOV2018}, for the initial conditions $e(\cdot,0) \in X$.

In this paper, we design sufficient conditions for the system \eqref{e:displacement-pde-delay}, with delay bound $\tau_M$, to achieve exponential stabilization (with any desirable decay rate for small enough $\tau_M\geq 0$).

In the following two sections, we consider the cases without (i.e., $\tau(t)=0$) and with delay in the communication channel, respectively. For both cases, we derive LMI conditions for desired decay rate with given system parameters $a, K$, and $\tau_M$ (the case with delay).

\section{PDE-based deployment}\label{s:periodic-nodelay}

Firstly, we shall consider the sampled-data controller without delay in this subsection, i.e., 
\begin{align}\label{e:displacement-pde-nodelay}
z_t = a(z_{xx}(x,t) - \gamma_{xx}) - K(z(\pi,t)-\gamma(\pi))
\end{align}
with periodic boundary condition \eqref{e:periodic}. In this case the dynamic of the error $e = z(x,t)-\gamma(x)$ is given as 
\begin{align}\label{e:errordyn-nodelay}
e_t = a e_{xx} - K e(\pi,t).
\end{align}
Similarly to the previous section, we assume $e\in\R$.

\begin{theorem}\label{th:periodic-nodelay}

For any $\delta>0$, let $K > \delta$ and $a \geq \frac{K^2}{K-\delta}$. Then the system \eqref{e:errordyn-nodelay}, \eqref{e:periodic_e} is exponentially stable with the decay rate $\delta$ in the $\calH^1$-norm:  
\begin{equation}
\exists\,C>0\colon\quad \|e(\cdot,t)\|_{\calH^1}\le Ce^{-2\delta t}\|e(\cdot,0)\|_{\calH^1}. 
\end{equation}

\end{theorem}

\begin{proof}
Since the components of $e$ are decoupled, here we prove the case that $e:[0,2\pi]\times \R_{\geq 0}\rightarrow \R$.
Consider the Lyapunov functional
\begin{align}\label{e:Lyapunov-nodelay}
V(t) =  \int_{0}^{2\pi} e^2(x,t) dx + q  \int_{0}^{2\pi} e^2_x(x,t) dx.
\end{align}
Let $\sigma:=e(x,t)-e(\pi,t)$. Then the system \eqref{e:errordyn-nodelay} can be written as 
\begin{align}
e_t = a e_{xx} - K e + K\sigma.
\end{align}
Then the time derivative of $V$ is given as 
\begin{align*}
\dot{V} = & 2 \int_0^{2\pi}e(a e_{xx} - K e + K\sigma) dx \\
& - 2q \int_{0}^{2\pi}e_{xx}(a e_{xx} - K e + K\sigma) dx \\
= & -2a\|e_x\|^2 - 2K\|e\|^2 + 2 K \int_{0}^{2\pi}e\sigma dx  \\
& - 2qa\|e_{xx}\|^2 - 2Kq\|e_x\|^2 - 2qK\int_0^{2\pi}e_{xx}\sigma dx,
\end{align*}
where we used the integral by parts.
By Lemma \ref{e:Wirtinger}, we have $\|\sigma\|^2 \leq 4 \|e_x\|^2$ which implies 
\begin{align}
0\leq -\lambda \|\sigma\|^2 + 4\lambda\|e_x\|^2, \quad \forall \lambda\geq 0.
\end{align}
Hence,
\begin{align*}
& \dot{V} + 2\delta V \\
\leq & -2a\|e_x\|^2 - 2K\|e\|^2 + 2 K \int_{0}^{2\pi}e\sigma dx  \\
& - 2qa\|e_{xx}\|^2 - 2Kq\|e_x\|^2 - 2qK\int_0^{2\pi}e_{xx}\sigma dx \\
& -\lambda \|\sigma\|^2 + 4\lambda\|e_x\|^2 + 2\delta\|e\|^2 + 2\delta q \|e_x\|^2 \\
= & \int_{0}^{2\pi} \eta^\top \Phi \eta dx + (4\lambda+2\delta q -2a-2Kq)\|e_x\|^2
\end{align*}
where $\eta = [e, e_{xx}, \sigma]^\top$ and  
\begin{align}
\Phi = \begin{bmatrix}
- 2K + 2\delta  &  0 &  K  \\
* & -2 qa & -q K  \\
*  & * & - \lambda  \\
\end{bmatrix}.
\end{align}
We have $\dot{V}+2\delta V\leq 0$ if $\Phi\leq 0$. By using Schur complement, $\Phi\leq 0$ if $K > \delta$ and
\begin{align*}
-\lambda + \begin{bmatrix}
K & -qK
\end{bmatrix}
\begin{bmatrix}
\frac{1}{2K-2\delta} & 0 \\ 0 & \frac{1}{2qa}
\end{bmatrix}
\begin{bmatrix}
K \\ -qK
\end{bmatrix}\leq 0.
\end{align*}
Taking $\lambda = \frac{a+Kq-\delta q}{2}$,  the last inequality is equivalent to 
\begin{align*}
\left(\frac{K^2}{a}-(K-\delta) \right) q \leq a - \frac{K^2}{K-\delta}. 
\end{align*}
Such $q$ exists if and only if $a \geq \frac{K^2}{K-\delta}$. Then $\dot{V}\leq -2\delta V$, which implies the exponential stability in the $\calH^1$-norm.

\end{proof}

\begin{remark}
If there are several leaders (e.g. at $\pi/2$ and $3/2 \pi$), then in \eqref{e:displacement-pde-nodelay}, we can use 
$-K(z(\pi/2,t)-\gamma (\pi/2))$ for $z\in[0,\pi)$,
and $-K(z(1.5 \pi,t)-\gamma (1.5\pi))$ for $z\in[\pi, 2\pi]$,
that allows to reduce the gain $a$ \cite{FRIDMAN2012}.
\end{remark}

%

\section{Network-based deployment}\label{s:periodic-delay}

In this subsection, we consider the stability of system \eqref{e:displacement-pde-delay} with bounded delay
and periodic boundary condition \eqref{e:periodic}. 

The main result is given as follows.

\begin{theorem}\label{th:periodic-delay}
Consider the boundary-value problem \eqref{e:dynamic+delay}, \eqref{e:periodic_e}. Given positive scalars $\delta_0,\tau_M, K$ and $\delta_1$ satisfying $\delta_1 < 2 \delta_0$, let there exist positive scalars $p_1,p_2,p_3,r,g$ and $s$ satisfying the following LMIs 
\begin{align}\label{e:LMI-delay}
  \delta_0 p_3 \leq p_2, \quad \Phi \leq 0, \quad
  \begin{bmatrix}
  r & s \\ s & r
  \end{bmatrix}\geq 0,
\end{align}
where $\Phi=\{\Phi_{ij}\}$ with
\begin{align*}
\Phi_{11} & = g + 2 \delta_0 p_1 - e^{-2 \delta_0 \tau_M} r, & \Phi_{12} & = p_1 - p_2, \\
\Phi_{14} & = -p_2 K + e^{-2 \delta_0 \tau_M} (r-s), &
\Phi_{13} & = e^{-2 \delta_0 \tau_M} s,\\
\Phi_{15} & = p_2 K, &
\Phi_{22} & = \tau_M^2 r -2 p_3, \\
\Phi_{24} & = -p_3 K, &
\Phi_{25} & = p_3 K, \\
\Phi_{33} & = -g e^{-2 \delta_0 \tau_M} -e^{-2 \delta_0 \tau_M} r, &
\Phi_{34} & = e^{-2 \delta_0 \tau_M} (r-s), \\
\Phi_{44} & = -2 e^{-2 \delta_0 \tau_M} (r-s) -\delta_1 p_1, &
\Phi_{55} & = -\frac{\delta_1 p_3 a}{4}.
\end{align*}
Then the system \eqref{e:dynamic+delay}, initialized with $e(x,t) = e(x,0)\in X, \forall t<0$, is exponentially stable with a decay rate $\delta$, where $\delta$ is the unique solution to \eqref{e:halanay-alpha}, in the $\calH^1$-norm:
\begin{align}
\exists\,C>0\colon\quad  \|e(\cdot,t)\|_{\calH^1} \leq Ce^{-2\delta t}\|e(\cdot,0)\|_{\calH^1} \label{e:convergence-inequality-delay}. 
\end{align}
 Moreover, given any $\delta>0$ and $K>\delta_0$, LMIs \eqref{e:LMI-delay} are always feasible for large enough $a$.
\end{theorem}

\begin{proof}
Consider the Lyapunov-Krasovskii functional 
\begin{align}
V(t) = & p_1\int_{0}^{2\pi} e^2(x,t) dx + p_3 \int_{0}^{2\pi} a e^2_{x}(x,t) dx \nonumber\\
& + \int_{0}^{2\pi}\bigg[ \tau_M r \int_{-\tau_M}^{0} \int_{t+\theta}^{t} e^{-2\delta_0(t-s)} e_s^2(x,s) ds d\theta \nonumber\\
& + g \int_{t-\tau_M}^{t} e^{-2\delta_0(t-s)} e^2(x,s) d s \bigg]dx. \label{e:Lyapunov-delay}
\end{align}
Notice that for the strong solution of \eqref{e:dynamic+delay}, the functional $V$ is well-defined and continuous.  
The time derivative of $V$ is given as 
\begin{align}
 \dot{V} + 2 \delta_0 V  
= & 2\delta_0 p_1  \int_{0}^{2\pi} e^2(x,t) + 2\delta_0 p_3 a \int_{0}^{2\pi} e^2_{x}(x,t) dx \nonumber\\
& + 2 p_1  \int_{0}^{2\pi} e_t e dx + 2 p_3 a \int_{0}^{2\pi} e_x e_{tx} dx \nonumber\\
& + g \int_{0}^{2\pi}[e^2(x,t)-e^{-2\delta_0\tau_M}e^2(x,t-\tau_M)]dx \nonumber\\
& -\tau_M r  \int_{0}^{2\pi}  \int_{t-\tau_M}^{t} e^{2\delta_0 (s-t)} e_s^2(x,s) ds dx \nonumber\\
& + \tau_M^2 r  \int_{0}^{2\pi} e^2_t(x,t) dx.\label{e:time-deri-V1}
\end{align}


By applying Jensen's inequality \cite[Proposition B.8]{gu2003stability} and further Park inequality (Lemma 1 in \cite{FRIDMAN2014}), we have 
\begin{align}
& -\tau_M r  \int_{0}^{2\pi}  \int_{t-\tau_M}^{t} e^{2\delta_0 (s-t)} e_s^2(x,s) ds dx \nonumber\\
\leq & -\frac{\tau_M}{\tau_M-\tau(t)}r e^{-2\delta_0 \tau_M} \int_{0}^{2\pi} [\int_{t-\tau_M}^{t-\tau(t)} e_s(x, s) ds]^2  dx \nonumber\\
& -\frac{\tau_M}{\tau(t)}r e^{-2\delta_0 \tau_M} \int_{0}^{2\pi} [\int_{t-\tau(t)}^{t} e_s(x, s) ds]^2  dx \nonumber\\
\leq & -e^{-2\delta_0 \tau_M} \int_{0}^{2\pi} \xi^\top 
\begin{bmatrix}
r & s \\ s & r
\end{bmatrix}
\xi dx  \label{e:time-deri-V2}
\end{align}
where $\xi^\top := [e(x, t)-e(x, t-\tau), e(x, t-\tau)-e(x, t-\tau_M)]$ and the parameter $s$ satisfies the last inequality of \eqref{e:LMI-delay}.

Due to \eqref{e:dynamic+delay}, we have 
\begin{align}
0 = & 2 \int_{0}^{2\pi} [p_2 e+p_3 e_t][ -e_t + a e_{xx} \nonumber \\
 & - K e(\pi,t-\tau(t)) ] dx. \label{e:D}
\end{align}
Using \eqref{e:time-deri-V2} in \eqref{e:time-deri-V1} and adding \eqref{e:D}, we find 
\begin{align*}
&\dot{V} + 2\delta_0 V - \delta_1 \sup_{\theta\in[-\tau_M,0]} V(t+\theta) \\
\leq & \dot{V} + 2\delta_0 V - \delta_1 V(t-\tau(t)) \\
\leq & \int_{0}^{2\pi} \varphi^\top \Phi \varphi dx - (2p_2 a - 2\delta_0 p_3 a) \int_{0}^{2\pi} e^2_x(x,t) dx
\end{align*}
with $\Phi$ given below \eqref{e:LMI-delay} and $\varphi^\top = [e(x,t),e_t(x,t),e(x,t-\tau_M),e(x,t-\tau),e(x,t-\tau)-e(\pi,t-\tau)]$.
This implies that, if $p_2\geq \delta_0 p_3$ and $\Phi \leq 0$, then \eqref{e:halanay-1} holds and 
Halanay's inequality implies that 
\begin{align}
V(t)\leq e^{-2\delta t} \sup_{\theta\in[-\tau_M,0]} V(\theta).
\end{align}
Finally, since the initial condition is set to be $e(x,t) = e(x,0), \forall t<0$, we have 
\begin{align}
& \sup_{\theta\in[-\tau_M,0]} V(\theta)\\
= & p_1\int_{0}^{2\pi} e^2(x,0) dx + p_3 \int_{0}^{2\pi} a e^2_{x}(x,0) dx  \nonumber\\
&  + g \int_{0}^{2\pi} \int_{-\tau_M}^{0} e^{2\delta_0 s} e^2(x,0) d s dx  \\
\leq & C\|e(\cdot,0)\|_{\calH^1}^2
\end{align}
where constant $C$ depends on the initial condition,
which implies the inequality \eqref{e:convergence-inequality-delay}. 

Now we show that the LMIs are feasible. Denote by $\Psi$ the matrix $\Phi$ with the deleted last column and row and $\delta_1=0$. Then $\Psi<0$ guarantees via the descriptor method that the system
$$\dot z(t)=-Kz(t-\tau)$$
is exponentially stable with a decay rate $\delta_0$ (cf. (4.23) in \cite{fridman2014introduction}). Moreover, given any $\delta_0>0$ and $K>\delta_0$ by arguments of \cite{fridman2014introduction} it can be shown that $\Psi<0$ is always feasible for small enough $\tau_M$.  Given any $\delta>0$ and choosing $\delta_1=0.1\delta$ and $\delta_0=\delta+2\delta_1$ and $K>\delta_0$, we find further $p_1,p_2,p_3,r$ and $s$ that solve $\Psi<0$ for small $\tau_M$. Then, applying Schur complements to the last column and row of $\Phi$, we conclude that $\Phi<0$ for large enough $a$. The latter implies that the system is exponentially stable with a decay rate $\delta$.

\end{proof}

\begin{remark}
For $\phi \in\calH^1(0,l)$, we have \cite{Kang2018}
\begin{align*}
\max_{x\in[0,l]} \phi^2(x) \leq 2\int_0^l \phi^2(\xi)d\xi + \int_0^l \phi^2_\xi (\xi) d\xi.
\end{align*}
Therefore, (24) implies 
\begin{equation*}
\exists C'>0\colon\quad \max_{x\in[0, 2\pi]} e^2(x,t) \leq C'e^{-2\delta t}\|e(0,t)\|_{\calH^1}^2.
\end{equation*}
\end{remark}

\section{Simulations}\label{s:simulations}

In this section, we present a simulative result of the proposed control laws in Section \ref{s:periodic-nodelay} and \ref{s:periodic-delay}. In the simulation, we consider a multi-agent system with $N=45$ agents. For the system parameters, we set $a = 10, K = 1$. In the figure of the deployment, the blue dashed lines are the desired formation, and the red dashed lines are the initial positions of the agents which are set to be $(0.5*\sin(i\frac{2\pi}{N}), 0.5*\cos(i\frac{2\pi}{N}), 0), i=1,\ldots,N$.  The black solid lines are the trajectories of the agents.

%

We start with the case without delay. Suppose the desired decay rate is $\delta = 0.6$. Notice that $K > \delta$ and $a \geq \frac{K^2}{K-\delta}$. Hence the decay rate is guaranteed. The performance of the system \eqref{e:displacement-pde-nodelay} is presented in Fig. \ref{fig:periodic_nodelay_x}. The first dimension of the error, i.e., $e_1(x,t)$, is depicted in Fig. \ref{fig:periodic_nodelay_e}. It can be seen that the error converges to zero.

\begin{figure}[t!]
\centering
\includegraphics[width=0.5\textwidth]{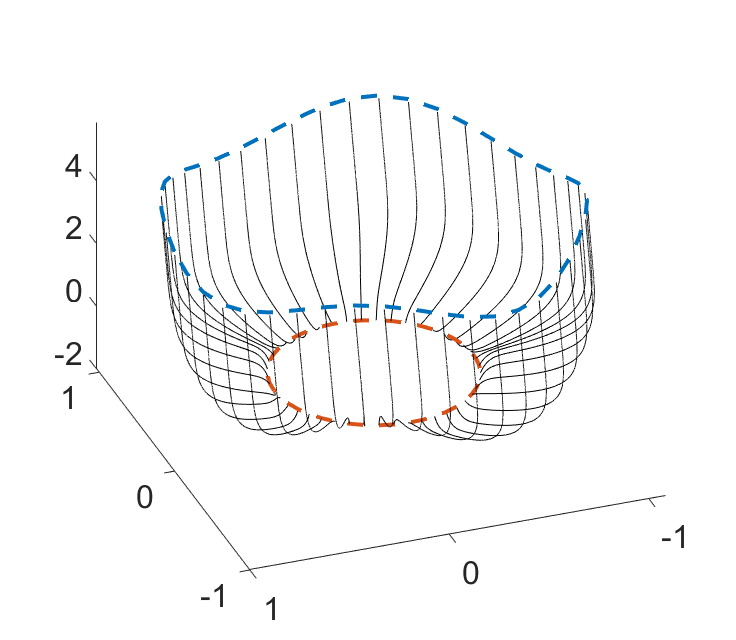}
\caption{Deployment of the agent according to the system \eqref{e:displacement-pde-nodelay} with periodic boundary condition \eqref{e:periodic}.}\label{fig:periodic_nodelay_x}
\end{figure}

\begin{figure}[t!]
\centering
\includegraphics[width=0.4\textwidth]{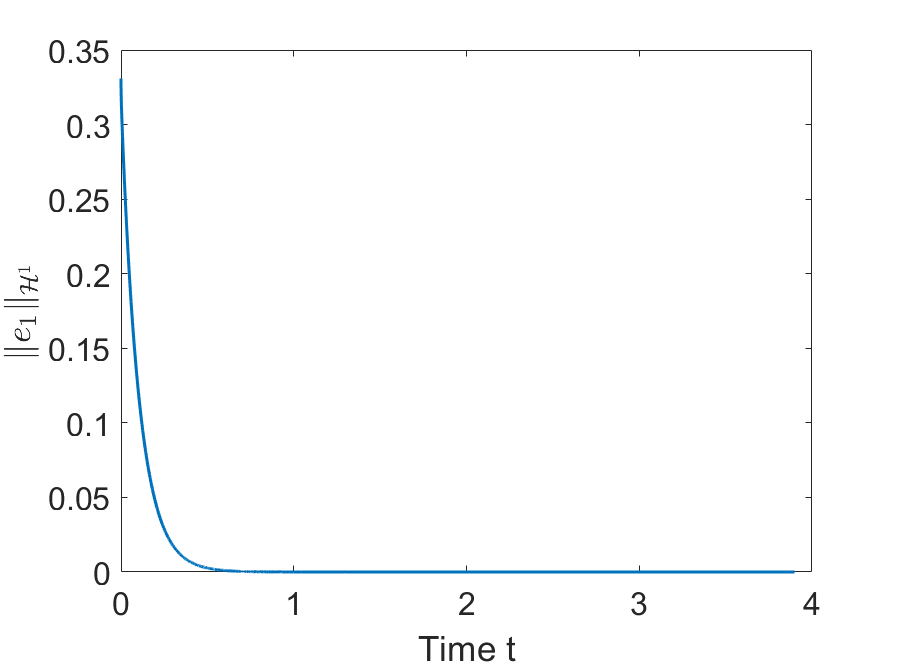}
\caption{The $\calH^1$ norm of error of the first dimension of the simulation given in Fig. \ref{fig:periodic_nodelay_x}.}\label{fig:periodic_nodelay_e}
\end{figure}

Next, we present an example for the case with delay. We choose $\delta_0=2.5$. Furthermore, the parameter $\delta_1$ in Halanay's inequality is set to be equal to $1.5\delta_0$ which is less than $2\delta_0$. The LMI conditions \eqref{e:LMI-delay} is satisfied by $p_1 = 0.19, p_2 = 0.32, p_3 = 0.12, r= 10, g=0.04, s=0.41$ and $\tau_M=0.01$ which is verified by CVX \cite{cvx}. This guarantees the same decay rate as without delay, i.e., $\delta= 0.60$. The performance of \eqref{e:displacement-pde-delay} is given in Fig. \ref{fig:periodic_delay_x} and the error of the first dimension is plotted in Fig. \ref{fig:periodic_delay_e}.

\begin{figure}[t!]
\centering
\includegraphics[width=0.5\textwidth]{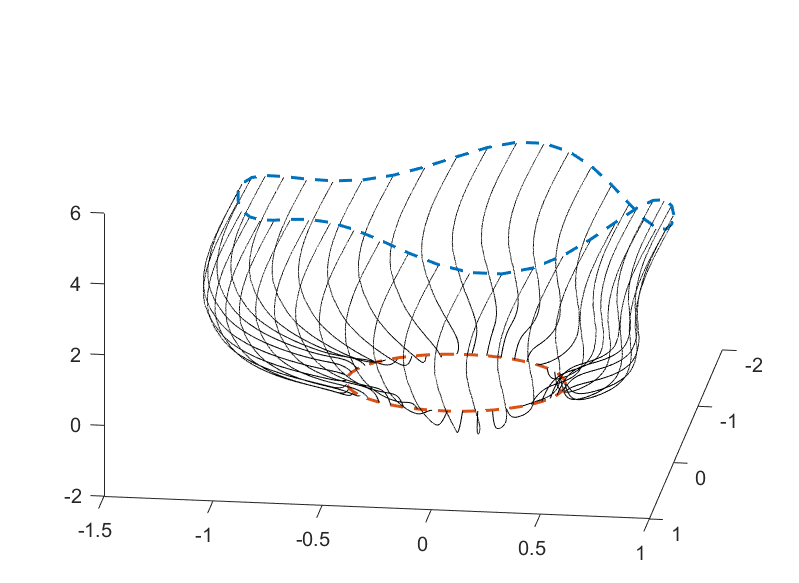}
\caption{Deployment of the agent according to the system \eqref{e:displacement-pde-delay} with periodic boundary condition \eqref{e:periodic}.}\label{fig:periodic_delay_x}
\end{figure}

\begin{figure}[t!]
\centering
\includegraphics[width=0.4\textwidth]{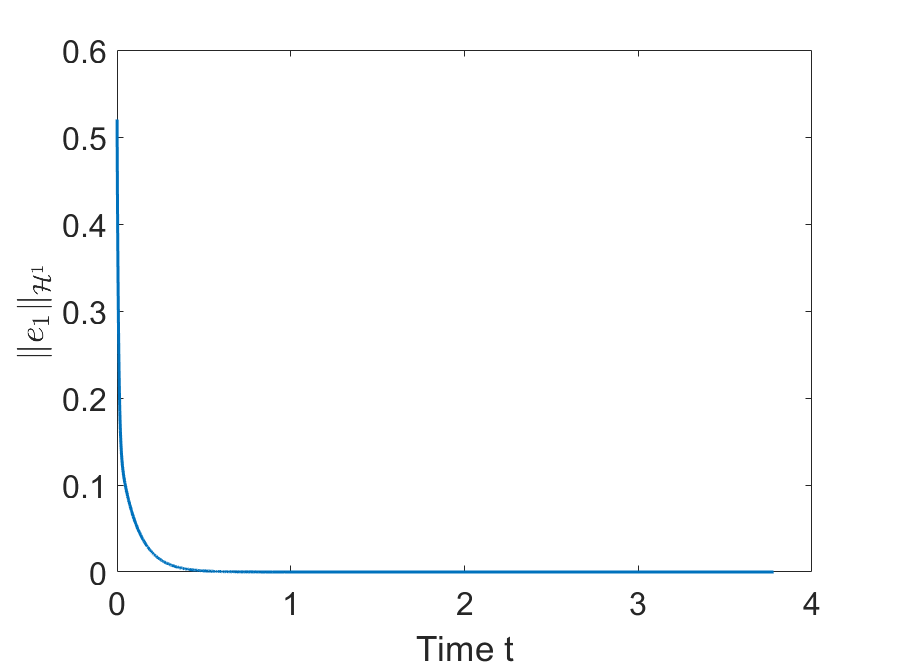}
\caption{The $\calH^1$ norm of error of the first dimension of the simulation given in Fig. \ref{fig:periodic_delay_x}.}\label{fig:periodic_delay_e}
\end{figure}

\section{Conclusion}\label{s:conclusion}

In this paper, we considered the deployment of the first-order multi-agents onto a desired closed smooth curve. The model is motivated by the displacement-based multi-agent formation control algorithm.  We assumed that the agents have access to the local information of the desired curve and their displacements with respect to their closest neighbors, whereas a leader is able to measure its absolute displacement with respect to the desired curve and transmit it to other agents through communication network. It was proved that, based on LMI conditions, by choosing appropriate controller gains, any desired decay rate can be achieved provided the delay is small enough. More precisely, exponential convergence to any closed $\calC^2$ curve is guaranteed.

\bibliographystyle{plain} 
\bibliography{ref,C:/Users/bartb/Documents/Literature/all}

\end{document}